\documentclass[11pt,twoside]{article}
\usepackage[dvips]{graphicx}
\usepackage{fancyhdr}
\usepackage{color}
\usepackage{amsmath}
\usepackage{amssymb}
\usepackage{bm}
\usepackage{latexsym}
\usepackage{arydshln}

 \numberwithin{equation}{section}

\definecolor{darkcyan}{rgb}{0.0, 0.55, 0.55}

\definecolor{darkviolet}{rgb}{0.58,0,0.83} 

\newcommand{\Hil}[0]{
\mathcal{H}
}
\newcommand{\norm}[2]{
\left\| #2 \right\|_{#1}
}

\newcommand{\HS}[0]{
{\mathcal HS}
}

\newcommand{\RR}[0]{
\mathbb{R}
}

\newcommand{\BL}[1]{
{\mathcal B} \left( #1 \right)
}

\newtheorem{theorem}{Theorem}[section]
\newtheorem{definition}{Definition}[section]
\newtheorem{proposition}[theorem]{Proposition}
\newtheorem{lemma}[theorem]{Lemma}
\newtheorem{corollary}[theorem]{Corollary}

\newtheorem{conj.}[theorem]{Conjecture}
\newtheorem{Bsp.}{Example}[section]
\newcommand{\identity}[1]{\mathsf{id}_{ #1 }}

\newenvironment{proof}{\noindent \bf Proof: \rm}{$ \hspace{\stretch{1}} \Box $

\vspace{5mm}}
\newcommand{\ltfun}[2]{\sum \limits_{#1} \bigoplus #2}
\def\ltiv{\ltfun{i\in I}{V_i}}
\def\ltiw{\ltfun{i\in I}{W_i}}
\def\w{\widetilde{W}}

\newcommand{\diag}[1]{\mathsf{diag}\left( #1 \right)} 
\newcommand{\range}[1]{\mathsf{R}\left( #1 \right)} 
\newcommand{\kernel}[1]{\mathsf{N}\left( #1 \right)} 

\begin{document}
%
%
\title{\bf\vspace{-39pt}  Representation of Operators Using Fusion Frames }
%
%
\author{Peter Balazs, Mitra Shamsabadi, Ali Akbar Arefijamaal, Chilles Gardon}

%
%
%

%
%
\date{}

%
%
\maketitle
\thispagestyle{fancy}

%
%

\begin{abstract}
For finding the numerical solution of operator equations in many applications a decomposition in subspaces is needed. Therefore, it is necessary to extend the known method of matrix representation to the utilization of fusion frames.
In this paper we investigate this representation of operators on a Hilbert space $\Hil$ with Bessel fusion sequences, fusion frame and Riesz decompositions.
We will give the basic definitions. We will show some structural results and give some examples.
Furthermore, in the case of Riesz decompositions, we prove that those functions are isomorphisms. Also, we want to find the pseudo-inverse and the inverse (if there exists) of such matrix representations. We are going to apply this idea to the Schatten $p$-class operators.  Finally, we show that tensors of fusion frame are frames in the space of Hilbert-Schmidt operators.
\vspace{5mm} 
\vspace{3mm}
\end{abstract}
\section{Introduction}
Frames were introduced by Duffin and Schaefer \cite{duffschaef1} and became a popular topic of mathematical research with the rise of wavelets \cite{daubgromay86}.
Frames have been the focus of active research, both in theory \cite{cach97,ehhan08,gr91} and applications \cite{Cai12,GrLe04,framepsycho16,boelc1,colgaz10}.
Frame theory has given rise to deep conjectures and results such as the Feichtinger conjecture \cite{gr03-4,maspsr15}.
 Also several generalizations have been investigated, e.g. \cite{antoin2,jpaxxl09,spexxl14,Sun2006437}, among them fusion frames \cite{caskut04,cakuli08,ga07}, which are the topic of this paper. The reason why frames became more and more important was that it can be hard to find a 'good' orthonormal basis, in the sense that it sometimes cannot fulfill given chosen  properties, as formally expressed e.g. in the Balian-Low theorem \cite{gr01} for Gabor frames, or some No-Go theorems for wavelets \cite{daubech1}.

This is also the reason why frames have become more popular in describing operators.
The numerical treatment  of operator equations, $Of =g$, requires a
discrete formulation, ${M} \vec{f} = \vec{g}$. This is often done with a
so-called  Galerkin scheme \cite{sauter2010boundary}, i.e. choosing atoms $\left( \psi_k \right)$ and look at matrices $\left< O \psi_k, \psi_lk \right>$.
In the Finite Element Method
\cite{brennscott1} and the Boundary Element Method \cite{Gauletal03}  usually spline-like
bases are used. Recently, wavelet bases \cite{DahSch99a} and frames \cite{Stevenson03,harbr08} have been applied.

On an abstract level, it is well known that for orthonormal sequences operators can
be uniquely described by a matrix representation \cite{gohberg1}.
An analogous result holds for frames and their duals \cite{xxlframoper1,xxlgro14}.
For solving this problem, frames are widely developed by many authors \cite{befe01,ch08,daubgromay86,feistro1}.

In approaches solving operator equations numerically one of the problems is the splitting of the considered spaces. Domain decomposition methods \cite{Werner09} solve a boundary value problem by splitting it into smaller boundary value problems on subdomains. Fusion frames allow the combination of solution on subspaces in a natural way. A first link of the concepts of space splittings and fusion frames was done in \cite{oswald09}. A first approach towards the matrix representation was done in \cite{sharxxl17}, where a view-point of a Gram-like operator was taken.

In this paper, we investigate this approach from a frame theory point of view, generalizing the approach in \cite{gohberg1} (for bases) and \cite{xxlframoper1} (for Hilbert frames). We settle the basic properties for a matrix representation of operators using fusion frames. In fact, we represent an operator by fusion frames and show that this representation is unique if the fusion frames are fusion Riesz bases.  We will settle the question how the invertibility of operators and matrices are related. In Section 2 we review basic notations and collect needed results. In Section 3 we give matrix representation of operators by Bessel fusion sequence and fusion frames. For an operator $O$ we obtain a matrix induced by two Bessel fusion sequences and conversely, associated with a matrix $M$ we define an operator induced by the matrix $M$ with respect to the Bessel fusion sequences. In Section 4 we obtain necessary and sufficient conditions for the invertibility  and pseudo-invertibility of such matrix representations. We restricted our attentions to Schatten $p-$class of operators in Section 5. We investigate the properties of matrix representations for Hilbert-Schmidt class of operator and show that for fusion frames $W$ and $V$, the space $(W\otimes V)$ constitutes a fusion frame for $S_2(\Hil)$, the set of Hilbert-Schmidt operators.

\section{Preliminaries}
Throughout this paper, $\pi_{W}$ denotes the orthogonal projection from $\mathcal{H}$
onto a closed subspace $W,$ $I_{\mathcal{H}}$  the identity operator on $\mathcal{H}$  and $\{e_{i}\}_{i\in I}$  an orthonormal basis for $\mathcal{H}$. By $\BL{\Hil_1,\Hil_2}$ we denote the space of all bounded operators between Hilbert spaces $\Hil_1$ and $\Hil_2$ and  write $\BL{\Hil}$ for $\Hil_1=\Hil_2=\Hil.$ Also, we represent the range and null space of a bounded operator $O\in \BL{\Hil_1,\Hil_2}$ by $\range{O}$ and $\kernel{O}$, respectively. Moreover, $O^{\dag}$ is the pseudo-inverse of    bounded and closed range operator  $O\in \BL{\Hil_1,\Hil_2}.$

\subsection{Fusion frames}

For each sequence $(W_{i},\omega_{i})$ of closed subspaces in $\mathcal{H}$, the space
\begin{equation*}
\left(\sum_{i\in I}\bigoplus W_{i}\right)_{\ell^{2}}=\left\{\{f_{i}\}_{i\in I}: f_{i}\in W_{i}, \sum_{i\in I}\|f_{i}\|^{2}<\infty \right\},
\end{equation*}
with the inner product
\begin{equation*}
\bigg\langle\{f_{i}\}_{i\in I},\{g_{i}\}_{i\in I} \bigg\rangle=\sum_{i\in I}\langle f_{i},g_{i}\rangle,
\end{equation*}
is a Hilbert space.

We now give the central definition of fusion frames:
\begin{definition}
Let $\{W_{i}\}_{i\in I}$ be a family of closed subspaces of $\mathcal{H}$ and $\{\omega _{i}\}_{i\in I}$
be a family of weights, i.e. $\omega_{i}>0, i\in I$. The sequence  $W=(W_{i},\omega_{i}
)$
is called a fusion frame for $\mathcal{H}$ if there exist constants $0<A_{W}\leq B_{W}<\infty$ such that
\begin{equation*}
A_{W}\|f\|^{2}\leq\sum_{i\in I}\omega _{i}^{2}\|\pi_{W_{i}}f\|^{2}\leq B_{W}\|f\|^{2}, \qquad (f\in \mathcal{H}).
\end{equation*}
The constants $A_{W}$ and $B_{W}$ are called \textit{fusion frame
bounds}. If we only assume the upper bound, we call
$(W_{i},\omega_{i})$ a $\textit{Bessel fusion
sequence}$. A fusion frame is called $\textit{tight}$, if $A_{W}$
and $B_{W}$ can be chosen to be equal, and $\textit{Parseval}$ if
$A_{W}=B_{W}=1$. If $\omega_{i}=\omega$ for all $i\in I$, the
collection $(W_{i},\omega_{i})$ is called
$\textit{$\omega$-uniform}$ and we abbreviate 1- uniform fusion
frames as $\{W_{i}\}_{i\in I}$. A fusion frame
$(W_{i},\omega_{i} )$ is said to be a $
\textit{fusion orthonormal basis}$ if $\mathcal{H}=\bigoplus_{i\in
I}W_{i}$ and it is called  a $\textit{Riesz decomposition}$ of
$\mathcal{H}$ if for every $f\in \mathcal {H} $ there is a unique
choice of $f_{i}\in W_{i}$ such that $f=\sum_{i\in I}f_{i}$.
\end{definition}
It is clear that every fusion orthonormal basis is a Riesz decomposition for $\mathcal{H}$, and also every Riesz decomposition is a 1-uniform fusion frame for $\mathcal{H}$ \cite{caskut04}.
Moreover, a family $\{W_{i}\}_{i\in I}$ of closed subspaces of $\mathcal{H}$
 is a fusion orthonormal basis if and only if it is
 a 1-uniform Parseval fusion frame \cite{caskut04}.\\

\noindent
Note that $S_W=id_\Hil$ if and only if $W$ is a fusion orthonormal basis.
In contrast, for Hilbert frames $S_\Psi=id_\Hil$ if and only if $\Psi$ is a Parseval frame.\\

The \textit{synthesis operator}
 $T_{W}:(\sum_{i\in I}\bigoplus W_{i})_{\ell^{2}}\rightarrow \mathcal{H}$
for a Bessel fusion sequence $(W_{i},\omega_{i})$ is
defined by
\begin{equation*}
T_{W}(\{f_{i}\}_{i\in I})=\sum_{i\in I}\omega_{i}f_{i}, \qquad (\{f_{i}\}_{i\in I}\in \sum_{i\in I}\oplus W_{i}).
\end{equation*}
The adjoint operator $T^{*}_{W}: \mathcal{H}\rightarrow (\sum_{i\in I}\bigoplus W_{i})_{\ell^{2}}$ which is called the $\textit{analysis operator}$
is given by
\begin{equation*}
T^{*}_{W}f=\{\omega_{i}\pi _{W_{i}}f\}_{i\in I}, \qquad  (f\in \mathcal{H}).
\end{equation*}
Both are bounded by $\sqrt{B_W}$.

If $W=(W_{i},\omega_{i}
)$ is a fusion frame, the $\textit{fusion frame operator}$ $S_{W}:\mathcal{H}\rightarrow \mathcal{H}$, which is defined by $S_{W}f=T_{W}T^{*}_{W}f=\sum_{i\in I}\omega_{i}^{2}\pi _{W_{i}}f$,
 is  bounded (with bound $B_W$), invertible and positive \cite{caskut04,ga07}.

Every Bessel fusion sequence
 $(V_{i},\upsilon_{i})$ is called a $\textit{G\v{a}vru\c{t}a-dual}$
  of  $(W_{i},\omega_{i})$, if
\begin{equation}\label{sec:reconstr1}
f=\sum_{i\in I}\omega_{i}\upsilon_{i}\pi_{V_{i}}S_{W}^{-1}\pi_{W_{i}}f,
\qquad (f\in \mathcal{H}),
\end{equation}
 for more details see \cite{caskut04,ga07}. From here on, for simplicity we say dual instead of  G\v{a}vru\c{t}a-dual.  The sequence of subspaces $\widetilde{W} := \left\{\left(S_{W}^{-1}W_{i},\omega_{i}\right)\right\}_{i\in I}$, which is a fusion frame for $\mathcal{H}$, is called the \textit{canonical dual} of $W$.

A Bessel fusion sequence $(V_{i},\upsilon_{i})$ is a dual of
fusion frame $W=(W_{i},\omega_{i})$ if and only if
\begin{equation} \label{sec:duality}
T_{V}\phi_{VW}T^{*}_{W}=\identity{\mathcal{H}} ,
\end{equation}
where the bounded operator $\phi_{VW}:(\sum_{i\in I}\bigoplus
W_{i})_{\ell^{2}}\rightarrow (\sum_{i\in I}\bigoplus
V_{i})_{\ell^{2}}$ is given by
\begin{equation}\label{phi}
\phi_{VW}(\{f_{i}\}_{i\in I})=\{\pi_{V_{i}}S_{W}^{-1}f_{i}\}_{i\in I}
\end{equation}
  and $\left\|\phi_{VW}\right\|\leq\left\|S_{W}^{-1}\right\|$.

Another approach to duality \cite{hemo14,behemoza14} uses a fixed arbitrary bounded operator $M:\ltiw \rightarrow \ltiv$. Starting with two fusion frames the duality is defined analogously to \eqref{sec:duality}, i.e.
$T_{V} \, M \, T^{*}_{W}=\identity{\mathcal{H}}$.
 We stick to the G\v{a}vru\c{t}a duals, but all results herein can be adapted to this other definition of duality.

 Let $\{W_{i}\}_{i\in I}$ be a family of closed subspaces of $\mathcal{H}$ and $\{\omega_{i}\}_{i\in I}$  a family of weights. We say that $(W_{i}, \omega_{i})$ is a \textit{fusion Riesz basis} for $\mathcal{H}$ if
 $\overline{\textrm{span}}\{W_i\}=\mathcal{H}$ and there exist constants $0<C\leq D<\infty$ such that for each finite subset $J\subseteq I$
\begin{equation}\label{riesz}
C\sum_{j\in J}\|f_{j}\|^{2}\leq \|\sum_{j\in J}\omega_{j}
f_{j}\|^2\leq D\sum_{j\in J}\|f_{j}\|^{2},  \qquad (f_{j}\in W_{j}).
\end{equation}

The next proposition explores fusion Riesz bases with respect to local frames and their operators.

\begin{proposition}\cite{caskut04,sharxxl17} \label{equi-Riesz}
Let $(W_i,w_i)$ be a family of closed subspaces and $\{e_{ij}\}_{j\in J_{i}}$ be an orthonormal basis for $W_{i}$ for each $i\in I$. Then the following conditions are equivalent:
\begin{enumerate}
\item[(1)] $(W_i,w_i)$ is a Riesz decomposition of $\mathcal{H}$.
\item[(2)] The synthesis operator $T_{W}$ is bounded and bijective.
\item[(3)] The analysis operator $T^{*}_{W}$ is bounded and bijective.
\item[(4)] $(W_i,w_i)$ is a fusion Riesz basis for $\mathcal{H}$.
\item[(5)] $\{w_ie_{ij}\}_{i\in I, j\in J_{i}}$ is a Riesz basis for $\mathcal{H}$.
\end{enumerate}
\end{proposition}

The following characterizations of fusion Riesz bases will be used frequently in this note.
\begin{proposition}
\label{R}\cite{mitra,sharxxl17}
Let $W=(W_{i},w_i)$ be a fusion frame in $\mathcal{H}$. Then the following are equivalent:
\begin{enumerate}
\item $W$ is a fusion Riesz  basis.
\item $S_{W}^{-1}W_{i}\perp W_{j}$ for all $i, j \in I, i\neq j$.
\item $\omega_{i}^2\pi_{W_{i}}S_{W}^{-1}\pi_{W_{j}}=\delta_{ij}\pi_{W_{j}}$, for all $i, j \in I$.
\end{enumerate}
\end{proposition}

\subsection{Tensor Product of Operators} \label{sec:tensop0}

Let $S \in \BL{\Hil_3,\Hil_4}$ and $T \in \BL{\Hil_1,\Hil_2}$, the tensor product of two operators as an element of $BL{\BL{\Hil_1,\Hil_3},\BL{\Hil_2,\Hil_4}}$
 is defined as follows
\begin{equation}
(S\otimes T)(O):= S \circ O \circ T^\ast,\qquad (O\in \BL{\Hil_1,\Hil_3}).
\end{equation}

The basic properties of tensor product of operators, see e.g. \cite{defant}, can be summarized in the following:
\begin{enumerate}
\item[(1)] $(S\otimes T)^*=T^*\otimes S^*$.
\item[(2)] $\left\|S\otimes T\right\|=\|S\|\|T\|.$
\item[(3)] $S\otimes T$ is   injective, surjective, respectively invertible if and only if $S$ and $T$ are injective, surjective, respectively invertible. In the later case, $(S\otimes T)^{-1}=S^{-1}\otimes T^{-1}$,
\end{enumerate}
 when  $ S\in \BL{\Hil_1,\Hil_4}$ and $ T\in \BL{\Hil_1,\Hil_2}$. Moreover, if $\Hil_1=\Hil_2=\Hil_4=\Hil$, then
\begin{enumerate}
\item[(4)] $\identity{\Hil}\otimes \identity{\Hil}= \identity{\Hil}$.
\item[(5)] $(S\otimes T)(A\otimes B)=(SA)\otimes (TB).$
\end{enumerate}



\section{Matrix Representations} \label{sec:descropfram0}
For orthonormal sequence it is well known that operators can be uniquely described by a matrix representation \cite{gohberg1}.
For sequences the matrix representation using Bessel sequences, frames, Riesz bases and orthonormal bases have been investigated theoretically in \cite{xxlframoper1}. The matrix representation of frames has also been used for the numerical treatment of operator equations \cite{dafora05,harbr08,Stevenson03}.

Here we extend the concept to Bessel fusion sequences and fusion frames. For that let $(W_i,w_i)$ and $(V_i,v_i)$  be two sets of  closed subspaces of $\Hil$ and $B_{j,i} : V_i \rightarrow W_j$ is a bounded operator. Define the \emph{block-matrix of operators} \cite{MANA:MANA19941670102} $\mathbf{B} : \ltiv \rightarrow \ltiw$ as
\begin{equation} \label{eq:blockmatrixop}
\mathbf{B f}:= \sum \limits_i B_{j,i} f_i,
\end{equation}
where $\mathbf{f}=\{f_i\}_{i\in I}\in \ltiv$.\footnote{This could be called a generalized subband matrix, motivated by system identification applications \cite{PerfAnSBI}.}
This extends the definition of the operator defined by a (possibly infinite) matrix: $\left( M c\right)_j = \sum \limits_k M_{j,k} c_k$.
Any operator in $\BL{\ltiv,\ltiw}
$ can be represented as matrix of operators. (See \cite{Maddox:101881,kohl19}.)

We will start with the more general case of Bessel fusion sequences. Note that we will use the notation $\norm{\Hil_1 \rightarrow \Hil_2}{.}$ for the operator norm in $\BL{\Hil_1, \Hil_2}$ to be able to distinguish between different operator norms.

\subsection{Matrix Representation for Bessel fusion Sequences}

Assume that $W=(W_i,w_i)$ and $V=(V_i,v_i)$ are  Bessel fusion sequences in $\Hil_1$ and $\Hil_2$, respectively. Using the tensor product of operators, for any operator $O\in \BL{\Hil_1,\Hil_2}$  we can define a matrix operator in  $\BL{\ltiv,\ltiw}$ by
\begin{eqnarray}\label{tensor}\left(\mathcal{M}^{(W,V)}(O)\right)_{i,j}=w_jv_i\pi_{W_j}O\pi_{V_i}, \qquad (i,j\in I).
\end{eqnarray}

\begin{theorem} \label{sec:matbyfram1} Let $W = (W_i,w_i)$ and $V = (V_i,v_i)$ be Bessel fusion sequences in Hilbert spaces $\Hil_1$ and $\Hil_2$ with  bounds $B_W$ and $B_V$, respectively.
\begin{enumerate} \item Let $O : \Hil_1 \rightarrow \Hil_2$ be a bounded linear operator. Then the matrix of operators \eqref{tensor}, denoted by
 $ \mathcal{M}^{(W , V)}(O),$
defines a bounded operator from $\ltiv$ to $\ltiw$ with
\begin{eqnarray*}
\norm{\ltiv \rightarrow \ltiw}{\mathcal{M}^{(W,V)}\left( O \right)} \le \sqrt{B_W \cdot B_V} \cdot \norm{\Hil_1 \rightarrow \Hil_2}{O}.
\end{eqnarray*}
As an operator $\ltiv \rightarrow \ltiw$
 \begin{equation}\label{DefM}
 {\mathcal M}^{(W , V)} \left( O \right) = T^\star_{W} O T_V=\left(T_W^*\otimes T_V^*\right)(O).
\end{equation}
This means the function ${\mathcal{M}}^{(W , V)} : \BL{\Hil} \rightarrow \BL{\ltiv,\ltiw}$ is a well-defined bounded operator.
\item On the other hand, let $M$ be an infinite matrix of operators defining a bounded operator from $\ltiv$ to $\ltiw$, $\left(M f\right)_i = \sum \limits_k M_{i,k} f_k$. Then the operator $\mathcal{O}^(W,V)$ defined by
$$ \mathcal{O}^{(W,V)} (M) := T_{W} M T_{V}^{*} = \left( T_W\otimes T_V \right) (M)$$
is a bounded operator from $\Hil_1$ to $\Hil_2$ with
$$\norm{\Hil_1 \rightarrow \Hil_2}{\mathcal{O}^{(W,V)} \left( M \right)} \le \sqrt{B_W \cdot B_V} \norm{\ltiv \rightarrow \ltiw}{M}.$$

This means the function ${\mathcal O}^{(W , V)} : \BL{\ltiv,\ltiw} \rightarrow \BL{\Hil_1,\Hil_2}$ is a well-defined bounded operator.
\end{enumerate}
\end{theorem}
\begin{proof}
Let $O \in \BL{\Hil_1,\Hil_2}$. Then
\begin{eqnarray*}
\left( \mathcal{M}^{(W , V)} \left( O \right) f\right)_j &=& \sum \limits_k \left( \mathcal{M}^{(W , V)} \left( O \right) \right)_{j,k} f_k\\
 &=&  \sum \limits_k v_k w_j \pi_{W_j} O \pi_{V_k} f_k \\
&=& w_j \pi_{W_j} \sum \limits_k v_k O f_k \\
&=& w_j \pi_{W_j} O \sum \limits_k v_k  f_k = \left(\left(T^\star_W\otimes T_V\right) (O) f\right)_j.
\end{eqnarray*}
Hence,
\begin{eqnarray*}
 \norm{\ltiw}{\mathcal{M}^{(W , V)} f}& \le&  \sqrt{B_W B_V} \norm{\Hil_1 \rightarrow \Hil_2}{O} \norm{\ltiw}{f}.
\end{eqnarray*}
The second part can be proved.
\end{proof} For an operator $O$ and a matrix $M$ as in Theorem \ref{sec:matbyfram1}, we call ${\mathcal M}^{(W,V)}(O)$ the \em matrix induced by the operator $O$ \em with respect to the Bessel fusion sequences $W = (W_i,w_i)$ and $V = (V_i,v_i)$, 
and ${\mathcal O}^{(W,V)}(M)$ the \em operator induced by the matrix $M$ \em with respect to the Bessel fusion sequences $W$ and $V$.
For an example, suppose that
 $W=(W_i,w_i)$, $V=(V_i,v_i)$ be  fusion frames    and $W^u=(W_i,1)$, $V^u=(V_i,1)$. Then
$$ \diag{\mathcal M^{(W^u,V^u)}(S_V^{-1})} = \phi_{W,V}.$$

Trivially, we have
\begin{eqnarray}
\left( \mathcal{M}^{(W,V)}\left( O \right) \right)^* & = & \mathcal{M}^{(V,W)}\left( O^* \right) \label{eq:adjM} \\
\left( \mathcal{O}^{(W,V)} (M) \right)^* & = & \mathcal{O}^{(V,W)} (M^*).  \label{eq:adjO}
\end{eqnarray}

\subsection{Matrix Representation for Fusion Frames}

Let $W=(W_i,w_i)$ and $V=(V_i,v_i)$ be fusion frames with the frame operators $S_W$ and $S_V,$ respectively. Then we define operators by means of $ \mathcal{M}^{(W,V)}$ and $ \mathcal{O}^{(W,V)}$:
$$ {\mathcal M}_{\otimes}^{(W , V)} = \left( T^*_{W} \otimes T^*_V  \right) \circ \left( S_W^{-1/2} \otimes S_V^{-1/2} \right)=  \left(  \mathcal{M}^{(W,V)} \right) \circ \left( S_W^{-1/2} \otimes S_V^{-1/2} \right).$$
On the other hand
$$ \mathcal{O}_{\otimes}^{(W,V)} =  \left( S_W^{-1/2} \otimes \circ S_V^{-1/2} \right) \circ \left( T_{W} \otimes T_V \right) =  \left( S_W^{-1/2} \otimes  S_V^{-1/2} \right) \circ \left( \mathcal{O}^{(W,V)} \right).$$
 These operators have nicer properties than the original ones, see below.
\begin{theorem}
Let $W=(W_i,w_i)$ and $V=(V_i,v_i)$ be fusion frames. Then $ \mathcal{M}_{\otimes}^{(W,V)}$ and $ \mathcal{O}_{\otimes}^{(W,V)}$ are bounded operators, moreover,
\begin{eqnarray} \label{eq:operreconst1}
\mathcal{O}_{\otimes}^{(W,V)} \mathcal{M}_{\otimes}^{(W,V)} = \identity{\BL{\Hil}}.
\end{eqnarray}
\end{theorem}
\begin{proof}
Using Theorem \ref{sec:matbyfram1} we have
\begin{eqnarray*}
\norm{}{\mathcal{M}_\otimes^{(W,V)}} &\le & \norm{}{ \left(  \mathcal{M}^{(W,V)} \right) \circ \left( S_W^{-1/2} \otimes S_V^{-1/2} \right)}\\
&\le&\sqrt{B_W \cdot B_V}\norm{}{S_W^{-1/2}}\norm{}{S_V^{-1/2}}\le \sqrt{\frac{B_W \cdot B_V}{{A_W \cdot A_V}}}.
\end{eqnarray*}
Similarly,
$$\norm{}{\mathcal{O}_\otimes^{(W,V)}}  \le \sqrt{\frac{B_W \cdot B_V}{{A_W \cdot A_V}}}.$$
Moreover,
\begin{eqnarray*}
\mathcal{O}_{\otimes}^{(W,V)} \mathcal{M}_{\otimes}^{(W,V)}O&=& \left( S_{W}^{-1/2} \otimes S_{V}^{-1/2} \right) \mathcal{O}^{(W,V)} \mathcal{M}^{(W,V)} \left( S_{W}^{-1/2} \otimes S_{V}^{-1/2} \right)  O\\
&=& S_{W}^{-1/2} T_{W}T_{W}^* S_{W}^{-1/2}   O
S_{V}^{-1/2} T_{V}T_{V}^{*}S_{V}^{-1/2}=O,
\end{eqnarray*}for all $O\in \BL{\Hil}$.
\end{proof}


Notice that in \cite{sharxxl17} we investigated the $U$-cross Gram matrix, i.e. the operator
$\phi_{WV}T_W^*OT_{V}$, denoted  by $\mathcal{G}_{O,W,V}$. In the notation of the current paper we have $$\mathcal{G}_{O,W,V} = \phi_{WV} \mathcal{M}^{(W , V)} \left( O \right) = \left\{\pi_{V_i} S_W^{-1} \mathcal{M}^{(W , V)} \left( O \right)\right\}_{i\in I}.$$
There a reconstruction formula is shown:
$$ \left[ T_w \otimes \left( T_W^* S_W^{-1} \right) \right] \mathcal{G}_{O,W,V} = O. $$
Note the similarity and difference to $\mathcal{M}_{\otimes}^{(W , V)} \left( O \right) = \mathcal{M}^{(W , V)} \left( S_W^{-1/2} O S_V^{-1/2} \right)$. Besides the different definition, the scope of this paper is also different, see Theorem \ref{inv-riesz} and Theorem \ref{6} . Nevertheless let us summarize some of the results in \cite{sharxxl17} interesting for this manuscript using the new terminology.
If  $O=I_{\mathcal{H}}$, it is called \textit{fusion cross Gram matrix} and denoted by $\mathcal{G}_{W,V}$. We used $\mathcal{G}_{W}$ for $\mathcal{G}_{W,W}$; the so called \textit{fusion Gram matrix}. For a fusion orthonormal basis $W=(W_{i},\omega_i)$, it is shown that $\mathcal{G}_{W}=\phi_{WW} = I_{\sum_{i\in I}\bigoplus W_i}$. Moreover, $\mathcal{G}_{O,W,W}$ is invertible if and only if $W$ is a fusion  Riesz basis and $O$ is invertible, in this case,
$$\mathcal{G}_{O,W,W}^{-1} = \mathcal{G}_{S_{W^u}^{-1}O^{-1}S_{W^u}^{-1},W,W},$$ where $W^u = (W_i, 1)$ is the uniformization of $W$.
Similar results are obtained for $\mathcal{G}_{O,\widetilde{W},W}$,  $\mathcal{G}_{O,W,\widetilde{W}}$ and $\mathcal{G}_{O,\widetilde{W},\widetilde{W}}$. Also, if $O$ has closed range and $O^{\dagger}$ denote its pseudo inverse, then with some additional assumptions we obtain
 \begin{equation*}
\left(\mathcal{G}_{O,\widetilde{W},W}\right)^{\dagger}=\mathcal{G}_{O^{\dagger},\widetilde{W},W}, \qquad
\left(\mathcal{G}_{O,W,W}\right)^{\dagger}= \mathcal{G}_{L_WO^{\dagger}L_W^{-1},W,W},
\end{equation*}
where $L_W=T_W\phi_{WW}T_{W}^*$ is the so called "alternate fusion frame operator" of $W$. For the proofs we refer the reader to \cite{sharxxl17}.

\section{Invertibility}

 In this section we investigate some conditions for the invertibility of $\mathcal{O}^{(W,V)}$, $\mathcal{M}^{(W,V)}$, $\mathcal{O}_{\otimes}^{(W,V)}$, $\mathcal{M}_{\otimes}^{(W,V)}$ and the operators induced by a matrix $M$ and  operator $O$. We will see that their invertibilities, for some cases, are equivalent.

Specifically,  let $W=(W_i,w_i)$, $V=(V_i,v_i)$ be Bessel fusion sequences,  $O\in \BL{\Hil}$ and $M\in\BL{\ltiv,\ltiw}$. Then $\mathcal{M}^{(W,V)}(O)$ has left (resp. right) inverse implies that $V$ (resp. $W$) is fusion Riesz sequence. Similarly, $\mathcal{O}^{(W,V)}(M)$ has left (resp. right) inverse implies that $V$ (resp. $W$) is fusion frame.

As an easy consequence of the properties of the tensors of operators, see Subsection \ref{sec:tensop0} we get
\begin{corollary}\label{Mlr}
Let $W=(W_i,w_i)$ and $V=(V_i,v_i)$ be Bessel fusion sequences. Then we have the following properties:
\begin{enumerate}
\item The operator $\mathcal{M}^{(W,V)} = T_W^* \otimes T_V^*$ is surjective  if and only if $T_W^*$ and $T_V^*$ are surjective (i.e. $W$ and $V$ are fusion Riesz bases).
\item The operator $\mathcal{M}^{(W,V)} = T_W^* \otimes T_V^*$ is injective, if and only if $T^*_W$ and $T^*_V$ is injective (i.e. $W$ and $V$ are complete).
\item The operator $\mathcal{O}^{(W,V)} = T_W \otimes T_V$ is surjective if and only if $T_W$ and $T_V$ are surjective (i.e. $V$ and $W$ are fusion frames).
\item The operator $\mathcal{O}^{(W,V)} = T_W \otimes T_V$ is injective if and only if $T_W$ and $T_V$ are injective.
\end{enumerate}
\end{corollary}

\subsection{Banach-algebra properties}

For every fusion frame $W=(W_i,w_i)$, the function $\mathcal M^{(W,W)}$ is a Banach-algebra homomorphism between the algebra of bounded operators $\BL{\Hil}$ and the matrices of operators in $\BL{\ltiw}$ if and only if $W$ is a fusion orthonormal basis in $\Hil$.
Indeed, $\mathcal{M}^{(W,W)}$ is a homomorphism only if
\begin{eqnarray*}
\mathcal{M}^{(W,W)}(O_1)\mathcal{M}^{(W,W)}(O_2)&=& T_{W}^*O_1S_WO_2T_W\\
&=&T_{W}^*O_1O_2T_W\\
&=&\mathcal{M}^{(W,W)}(O_1O_2).
\end{eqnarray*}
In particular, $O_1S_WO_2=O_1O_2$. This easily follows that (take $O_1=O_2=\identity{\Hil}$) $S_W=\identity{\Hil}$ and so $W$ is a fusion orthonormal basis.

\begin{proposition} For every fusion frame $W = (W_i,w_i)$, the function $\mathcal M_{\otimes}^{(W,W)}$ is a $C^*$-algebra monomorphism between $\BL{\Hil}$ and the matrices of operators in $\BL{\ltiw}$. It is an isomorphism if and only if $W$ is a fusion Riesz basis.
\end{proposition}
\begin{proof}
Clearly, for all $O_1,O_2\in \BL{\Hil}$ we have
\begin{eqnarray*}
\mathcal M_{\otimes}^{(W,W)} \left( O_1  O_2 \right) & = & T_W^* S_W^{-1/2} O_1 O_2 S_W^{-1/2} T_W   \\
& = & T_W^* S_W^{-1/2} O_1 S_W^{-1/2} T_W T^*_W S_W^{-1/2} O_2 S_W^{-1/2} T_W  \\
& = & \mathcal M_{\otimes}^{(W,W)} \left( O_1 \right) \mathcal M_{\otimes}^{(W,W)} \left( O_2 \right).
\end{eqnarray*}
Hence, $\mathcal M_{\otimes}^{(W,W)}$ is a $C^*$-algebra homomorphism by using (\ref{eq:adjM}) and (\ref{eq:adjO}).
Furthermore,
 \begin{equation}\label{homo}
 \mathcal{M}_{\otimes}^{(W,W)} \left( \identity{\Hil} \right) = T_W^* S_W^{-1/2} \identity{\Hil} S_W^{-1/2} T_W.
\end{equation}
Hence, The rest follows from Corollary \ref{Mlr}.
\end{proof}
\begin{corollary}
Let $W=(W_i,w_i)$ be a fusion frame in $\Hil$, $O\in \BL{\Hil}$ and $M\in \BL{\ltiw}$. The following are equivalent:
\begin{enumerate}
\item $W$ is a fusion Riesz basis.
 \item If $O$ is invertible, then $\mathcal{M}_{\otimes}^{(W,W)}(O)$ is invertible and $\left(\mathcal{M}_{\otimes}^{(W,W)}(O)\right)^{-1}=\mathcal{M}_{\otimes}^{(W,W)}(O^{-1})$.
\item If $M$ is invertible, then $\mathcal{O}_{\otimes}^{(W,W)}(M)$ is invertible and $\left(\mathcal{O}_{\otimes}^{(W,W)}(M)\right)^{-1}=\mathcal{O}_{\otimes}^{(W,W)}(M^{-1})$.
\end{enumerate}
\end{corollary}

\subsection{Invertibility for Riesz Fusion Sequences}
Similar to the Hilbert frame case \cite{xxlrieck11,xxlriek11} we can show formulas for the inverse and pseudo-inverse of the matrix induced by operators.
They are not just a straightforward generalization, because  of the interesting situation for fusion duals.
In the next, we compute the inverse of operator matrix induced by the operator $O$ with respect to fusion Riesz bases $W$ and $V$. First we give the following proposition.
 \begin{proposition}
Let $W=(W_i,w_i)$ be a fusion frame in $\Hil$ and $O\in \BL{\Hil}$ an invertible operator. The following are equivalent:
\begin{enumerate}
\item $W$ is a fusion orthonormal basis.
\item $\mathcal{M}^{(W,W)}(\identity{\Hil})=\identity{\ltiw}$.
\item $\mathcal{M}^{(W,W)}(O)$ is invertible and $\left(\mathcal{M}^{(W,W)}(O)\right)^{-1}=\mathcal{M}^{(W,W)}(O^{-1})$.
\end{enumerate}
\end{proposition}
\begin{proof}
$(1 \Rightarrow 2)$
 Let $W$ be a fusion orthonormal basis, then $W_{i}\perp W_{j}$, for all $i\neq j$ by Proposition \ref{R}. Hence,
\begin{eqnarray*}
\mathcal{M}^{(W,W)}(\identity{\Hil})=T_{W}^{*}T_{W}=\identity{\ltiw}.
\end{eqnarray*}
$(2 \Rightarrow 1)$
Conversely, suppose that $\mathcal{M}^{(W,W)}(\identity{\Hil})=\identity{\ltiw}$. Then
\begin{eqnarray*}
S_{W}^{2}&=&T_{W}T_{W}^{*}T_{W}T_{W}^{*}\\
&=&T_{W}\mathcal{M}^{(W,W)}(\identity{\Hil})T_{W}^{*}\\
&=&T_{W}T_{W}^{*}=S_{W}.
\end{eqnarray*}
The invertibility of $S_{W}$ implies that $S_{W}=\identity{\Hil}$ and so $W$ is a fusion orthonormal basis.

$(1 \Rightarrow 3)$
Let $W$ be a fusion orthonormal basis. For an orthonormal basis $W$ we obtain
 \begin{eqnarray*}
 \mathcal{M}^{(W,W)}(O^{-1})\mathcal{M}^{(W,W)}(O)&=&T_{W}^*O^{-1}T_{W}T_{W}^*OT_{W}\\
 &=&T_{W}^{*}O^{-1}S_{W}OT_{W}\\
 &=&\mathcal{M}^{(W,W)}(\identity{\Hil})=\identity{\ltiw}.
  \end{eqnarray*}
Similarly, $ \mathcal{M}^{(W,W)}(O^{-1})$ is a right inverse of $ \mathcal{M}^{(W,W)}(O)$.

$(3 \Rightarrow 1)$
Conversely, if $\mathcal{M}^{(W,W)}(O)$ is invertible, then $W$ is a fusion Riesz basis by Theorem 4.4. Also, if
$\left(\mathcal{M}^{(W,W)}(O)\right)^{-1}=\mathcal{M}^{(W,W)}(O^{-1}),$ then
\begin{eqnarray*}
S_{W}^{-1}&=&\left(T_{W}^*\right)^{-1}\identity{\ltiw}\left(T_{W}\right)^{-1}\\
&=& \left(T_{W}^*\right)^{-1}\mathcal{M}^{(W,W)}(O^{-1})\mathcal{M}^{(W,W)}(O)\left(T_{W}\right)^{-1}\\
&=&O^{-1}T_{W}T_{W}^*O=O^{-1}S_{W}O.
\end{eqnarray*}
The above computations show that $S_{W}OS_{W}=O$. Putting $O:=S_{W}^{-1}$ we obtain $S_{W}=\identity{\Hil}$  since $S_{W}$ being positive. Hence,  $W$ is a fusion orthonormal basis which is equivalent to $\mathcal{M}^{(W,W)}(\identity{\Hil})=\identity{\ltiw}$.

\end{proof}
The next result states the invertibility  of some matrix operators when the operator $O\in \BL{\Hil}$ is not necessary invertible.
Those results might seem trivial, but note that $S_{W}^{-1} \pi_{W_i} \neq \pi_{W_i}S_{W}^{-1}$, which is the reason for many interesting properties of fusion duals see also (\ref{sec:reconstr1}).
\begin{theorem}\label{6}
Suppose that $W=(W_i,w_i)$ is a fusion frame in $\Hil$ and $O\in \BL{\Hil}$. Then
\begin{enumerate}
\item $\mathcal{M}^{(W,\widetilde{W})}(O)$ is invertible if and only if $\mathcal{M}^{(W,W)}(OS_{W}^{-1})$ is invertible.
\item $\mathcal{M}^{(\widetilde{W},W)}(O)$ is invertible if and only if $\mathcal{M}^{(W,W)}(S_{W}^{-1}O)$ is invertible.
\item  $\mathcal{M}^{(\widetilde{W},\widetilde{W})}(O)$ is invertible if and only if $\mathcal{M}^{(W,W)}(S_{W}^{-1}OS_{W}^{-1})$ is invertible.
\item  $\mathcal{M}_{\otimes}^{(W,\widetilde{W})}(O)$ is invertible if and only if $\mathcal{M}_{\otimes}^{(W,W)}(OS_{\widetilde{W}}^{-1/2}S_W^{-1/2})$ is invertible.
\item  $\mathcal{M}_{\otimes}^{(\widetilde{W},W)}(O)$ is invertible if and only if $\mathcal{M}_{\otimes}^{(W,W)}(S_{W}^{-1/2}S_{\widetilde{W}}^{-1/2}O)$ is invertible.
\item  $\mathcal{M}_{\otimes}^{(\widetilde{W},\widetilde{W})}(O)$ is invertible if and only if $\mathcal{M}_{\otimes}^{(W,W)}(S_{W}^{-1/2}S_{\widetilde{W}}^{-1/2}OS_{\widetilde{W}}^{-1}S_W^{-1/2})$ is invertible.
\end{enumerate}
\end{theorem}
\begin{proof}
1. First, we show that
\begin{eqnarray}\label{si}
T_{\widetilde{W}}=S_{W}^{-1}T_{W}(\oplus S_W),
\end{eqnarray}
where $\oplus S_W:\sum \limits_{i\in I} \bigoplus\widetilde{W}_{i}\rightarrow \ltiw$ is given by
\begin{eqnarray}\label{funcsi}
\oplus S_W (\mathbf{g})=\left\{S_{W}g_{i}\right\}_{i\in I},\qquad \left(\mathbf{g}=\{g_{i}\}_{i\in I}\in \sum \limits_{i\in I} \bigoplus\widetilde{W}_{i}\right).
\end{eqnarray}
More precisely, by (\ref{funcsi}) we obtain for $\{f_i\}_{i\in I} \in \sum_{i\in I}{\oplus W_i}$,
\begin{eqnarray*}
T_{\widetilde{W}}\mathbf{g}&=&T_{\widetilde{W}}\left\{S_{W}^{-1}f_i\right\}_{i\in I}\\
&=&S_{W}^{-1}\sum \limits_{i\in I}\pi_{Wi}f_{i}\\
&=&S_{W}^{-1}T_{W}\left\{S_{W}g_{i}\right\}_{i\in I}
=S_{W}^{-1}T_{W}(\oplus S_W) \mathbf{g},
\end{eqnarray*}
Furthermore, applying (\ref{si}) we have
\begin{eqnarray*}
\mathcal{M}^{(W,\widetilde{W})}(O)&=&T_{W}^{*}OT_{\widetilde{W}}\\
&=&T_{W}^{*}OS_{W}^{-1}T_{W}(\oplus S_W)=\mathcal{M}^{(W,W)}(OS_{W}^{-1})(\oplus S_W).
\end{eqnarray*}
So, (1)  follows immediately  by the invertibility of $\oplus S_W$.

2. To obtain the second part, note that
\begin{eqnarray*}
\mathcal{M}^{(\widetilde{W},W)}(O)&=&T_{\widetilde{W}}^{*}OT_{W}\\
&=&(\oplus S_W)^{*}T_{W}^{*}S_{W}^{-1}OT_{W}=(\oplus S_W)^{*}\mathcal{M}^{(W,W)}(S_{W}^{-1}O).
\end{eqnarray*}
Moreover, $(\oplus S_W)^{*}:\ltiw\rightarrow \sum \limits_{i\in I} \bigoplus\widetilde{W}_{i}$ is invertible, and hence (3) is  easily proved.


4. Using \eqref{funcsi} and the fact that
\begin{eqnarray*}
\mathcal{M}_{\otimes}^{(W,\widetilde{W})} O&=&\mathcal{M}^{(W,\widetilde{W})} \left(S_{W}^{-1/2}\otimes S_{\widetilde{W}}^{-1/2}\right) O\\
&=& T_W^*S_W^{-1/2}OS_{\w}^{-1/2}T_{\w}\\
&=& T_W^*S_W^{-1/2}OS_{\w}^{-1/2}S_W^{-1}T_W(\oplus S_W)
= \mathcal{M}_{\otimes}^{(W,W)}\left(OS_{\w}^{-1/2}S_W^{-1/2}\right)(\oplus S_W)
\end{eqnarray*}
follow  the result.

5. Applying the identity $T_{\w}=S_{W}^{-1}T_{W}(\oplus S_W) $, proven in (1), we obtain
\begin{eqnarray*}
\mathcal{M}_{\otimes}^{(\w,W)}O&=&\mathcal{M}^{(\w,W)}\left(S_{\w}^{-1/2}\otimes S_W^{-1/2}\right)O\\
&=&T_{\w}^*S_{\w}^{-1/2}OS_W^{-1/2}T_W\\
&=&(\oplus S_W)^*T_{W}^*S_W^{-1}S_{\w}^{-1/2}OS_W^{-1/2}T_W
=(\oplus S_W)^*\mathcal{M}_{\otimes}^{(W,W)}\left(S_{W}^{-1/2}S_{\w}^{-1/2}O\right).
\end{eqnarray*}
It immediately implies the result by the invertibility of $\oplus S_W$. The proof of (3) and (6) are similar.
\end{proof}
We are ready now to state the main result of this section.

\begin{theorem}\label{inv-riesz}
Let $W=(W_i,w_i)$, $V=(V_i,v_i)$ be fusion frames in $\Hil$. Then the following are equivalent:
\begin{enumerate}
\item  $W$ and $V$ are fusion Riesz bases.
\item  $\mathcal{M}^{(W,V)}$ is onto.
\item  $\mathcal{M}^{(W,V)}$ is invertible and
$
\left(\mathcal{M}^{(W,V)}\right)
^{-1}=\left(S_W^{-1}\otimes S_V^{-1}\right)\mathcal{O}^{(W,V)}.
$
\item  There exists an invertible operator $O\in \BL{\Hil}$ such that $\mathcal{M}^{(W,V)}(O)$ is invertible  and
\begin{equation*}
\left(\mathcal{M}^{(W,V)}(O)\right)^{-1}=\mathcal{M}^{(V,W)}(S_V^{-1}\otimes S_W^{-1})(O^{-1}).
\end{equation*}
\item  $\mathcal{M}_{\otimes}^{(W,V)}$ is onto.
\item  $\mathcal{M}_{\otimes}^{(W,V)}$ is invertible and $
\left(\mathcal{M}_{\otimes}^{(W,V)}\right)^{-1}= \mathcal{O}_{\otimes}^{(W,V)}.$
\item There exists an invertible operator $O\in \BL{\Hil}$ such that $\mathcal{M}_{\otimes}^{(W,V)}(O)$ is invertible and
\begin{equation*}
\left(\mathcal{M}_{\otimes}^{(W,V)}(O)\right)^{-1}=\mathcal{M}_{\otimes}^{(V,W)}(O^{-1}).
\end{equation*}
\end{enumerate}
\end{theorem}
 \begin{proof}
$(1\Leftrightarrow2\Leftrightarrow 3)$  follows from Corollary \ref{Mlr}.
The formula for the inverse is easy to show.

$ (3 \Rightarrow 4)$
Assume that $\mathcal{M}^{(W,V)}$ is invertible. So, by Corollary \ref{Mlr}, $W$ and $V$ are fusion Riesz bases. Moreover, using Proposition \ref{equi-Riesz} for a fusion Riesz basis $W$ easily follows that
\begin{eqnarray}\label{SRiesz}
T_{W}^{*}S_{W}^{-1}T_{W}=\identity{\ltiw}.
\end{eqnarray}
So,
\begin{eqnarray*}
\mathcal{M}^{(W,V)}(O)\mathcal{M}^{(V,W)}(S_V^{-1}O^{-1}S_{W}^{-1})&=&T_{W}^{*}OT_{V}T_{V}^{*}S_V^{-1}
O^{-1}S_{W}^{-1}T_{W}\\
&=&T_{W}^{*}S_{W}^{-1}T_{W}=\identity{\ltiw}.
\end{eqnarray*}
Moreover,
\begin{eqnarray*}
\mathcal{M}^{(V,W)}(S_V^{-1}O^{-1}S_{W}^{-1})\mathcal{M}^{(W,V)}(O)&=&T_{V}^{*}S_V^{-1}
O^{-1}S_{W}^{-1}T_{W}T_{W}^{*}OT_{V}\\
&=&T_{V}^{*}S_V^{-1}T_{V}
=\identity{\ltiv}.
\end{eqnarray*}
$ (4 \Rightarrow 3)$
    For the reverse, suppose  that $M^{-1}\in \BL{\ltiw,\ltiv}$ is the inverse of $\mathcal{M}^{(W,V)}(O)$. Then
\begin{eqnarray*}
\identity{\ltiv}=M^{-1}\mathcal{M}^{(W,V)}O=M^{-1}T_W^*OT_V
\end{eqnarray*}
and
$$\identity{\ltiw}=\mathcal{M}^{(W,V)}OM^{-1}=T_W^*OT_VM^{-1}$$
follow that $T_V$ and $T_W^*$ are injective and surjective, respectively, and so $W$ and $V$ are fusion Riesz bases by Proposition \ref{equi-Riesz}. The rest is clear by Theorem \ref{inv-riesz}.

 $(5) \Leftrightarrow (6) \Leftrightarrow (7)$     are similar.

$(2)\Leftrightarrow(5), (3)\Leftrightarrow(6)$ by the definition of $\mathcal{M}$ and $\mathcal{M}_\otimes$.

\end{proof}
The interesting aspect of the last result was, that from the invertibility for only one instance, the invertibility of the whole operator can be deduced.


\subsection{Pseudo-Inverse}
Suppose that $U:\Hil_1\rightarrow\Hil_2$ is a bounded linear
operator with closed range $\mathcal{R}(U)$. Then there exists a unique
bounded linear operator $U^{\dag}:\Hil_2\rightarrow \Hil_{1}$ \cite{ch08}
satisfying \begin{eqnarray*} \range{U^{\dag}}=\range{U^*},\qquad
 \kernel{U^{\dag}}=\kernel{U^*} ,\qquad UU^{\dag}U=U.
\end{eqnarray*} The operator $U^{\dag}$ is called the \textit{pseudo-inverse operator} of $U$.
If $U$ has closed range, then $U^*$ and $UU^*$ have
closed range and $\left(U^*\right)^{\dag}=\left(U^{\dag}\right)^*$ and
$\left(UU^*\right)^{\dag}=\left(U^*\right)^{\dag}U^{\dag}$.

Our goal of this subsection is to obtain the pseudo-inverse of the operators induced in the third section.  Here, we assume that $\mathcal{M}^{(W,V)}(O)$ and $\mathcal{M}_{\otimes}^{(W,V)}(O)$ have closed range, which is true, for example if $O$ is onto. The advantage of $\mathcal{M}_{\otimes}^{(W,V)}(O)$ over $\mathcal{M}^{(W,V)}(O)$ is that its (pseudo-inverse) representation is obtained without any condition.

\begin{theorem} Let $W$ and $V$ be fusion frames,
$O\in \BL{\Hil}$ and $M\in \BL{\ltiv,\ltiw}$ have closed range. Then the
following assertions hold.
\begin{enumerate}
\item If $\mathcal{M}^{(W,V)}(O)$ has  closed range, then $\left(\mathcal{M}^{(W,V)}(O)\right)^{\dag}=\mathcal{M}^{(V,W)}(S_V^{-1}O^{\dag}S_W^{-1})$   if and only if $S_V\range{O^*}=\range{O^*}$ and $S_W\range{O}=\range{O}$.
 \item  If $\mathcal{M}_{\otimes}^{(W,V)}(O)$ has  closed range, then $\left(\mathcal{M}_{\otimes}^{(W,V)}(O)\right)^{\dag}=\mathcal{M}_{\otimes}^{(V,W)}(O^{\dag})$.
\end{enumerate}
\end{theorem}
\begin{proof}
It is straightforward to see that
$$\mathcal{M}^{(W,V)}(O)\mathcal{M}^{(V,W)}(S_V^{-1}O^{\dag}S_W^{-1})
\mathcal{M}^{(W,V)}(O)=\mathcal{M}^{(W,V)}(O).$$
Clearly, $S_V\range{O^*}=\range{O^*}$ if and only if
\begin{eqnarray}\label{ra}\range{T_V^*S_V^{-1}O^{\dag}}=\range{T_V^*O^{*}}.\end{eqnarray}
So,
\begin{eqnarray*}
\range{\mathcal{M}^{(V,W)}(S_V^{-1}O^{\dag}S_W^{-1})}&=&\range{T_V^*S_V^{-1}O^{\dag}S_W^{-1}
T_{W}}\\
&=&\range{T_V^*S_V^{-1}O^{\dag}}\\
\text {(by using (\ref{ra}))}&=& \range{T_V^*O^{*}}\\
&=&\range{T_V^*O^{*}T_W}=\range{\left(\mathcal{M}^{(W,V)}\left(O\right)\right)^*}.
\end{eqnarray*}
Also, $S_W\range{O}=\range{O}$ if and only if
\begin{eqnarray}\label{ke}
\kernel{O^{\dag}S_W^{-1}
T_{W}}=\kernel{O^*T_W}
\end{eqnarray}
\begin{eqnarray*}
\kernel{\mathcal{M}^{(V,W)}(S_V^{-1}O^{\dag}S_W^{-1})}&=&\kernel{T_V^*S_V^{-1}O^{\dag}S_W^{-1}
T_{W}}\\
&=&\kernel{O^{\dag}S_W^{-1}
T_{W}}\\
\text {(by using (\ref{ke}))}&=&\kernel{O^*T_W}\\
&=&\kernel{T_V^*O^*T_W}=\kernel{\left(\mathcal{M}^{(W,V)}(O)\right)^*}.
\end{eqnarray*}
Hence, $\left(\mathcal{M}^{(W,V)}(O)\right)^{\dag}=\mathcal{M}^{(V,W)}(S_V^{-1}O^{\dag}S_W^{-1})$.

2.
One can see that $$\mathcal{M}_{\otimes}^{(W,V)}(O)\mathcal{M}_{\otimes}^{(V,W)}(O^{\dag})\mathcal{M}_{\otimes}^{(W,V)}(O)=\mathcal{M}_{\otimes}^{(W,V)}(O).$$
Also,
\begin{eqnarray*}
\range{\left(\mathcal{M}_{\otimes}^{(W,V)}(O)\right)^{\dag}}&=&\range{T_V^*S_V^{-1/2}O^{\dag}}\\
&=&\range{T_V^*S_V^{-1/2}O^{*}}\\
&=&\range{T_V^*S_V^{-1/2}O^{*}S_W^{-1/2}T_W}=\range{\left(\mathcal{M}_{\otimes}^{(W,V)}(O)\right)^{*}}.\\
\end{eqnarray*}
and
\begin{eqnarray*}
\kernel{\left(\mathcal{M}_{\otimes}^{(W,V)}(O)\right)^{\dag}}&=&\kernel{O^{\dag}S_W^{-1/2}T_W}\\
&=&\kernel{T_VS_V^{-1/2}O^{*}S_W^{-1/2}T_W}=\kernel{\left(\mathcal{M}_{\otimes}^{(W,V)}(O)\right)^{*}}.
\end{eqnarray*}
Therefore, $\left(\mathcal{M}_{\otimes}^{(W,V)}(O)\right)^{\dag}=\mathcal{M}_{\otimes}^{(V,W)}(O^{\dag})$.
\end{proof}

\section{Matrix representation of Schatten $p$-class operators}

 Given $0<p<\infty$, we
define the Schatten $p$-class of $\mathcal{H}$, denoted by
$S_{p}(\mathcal{H})$, as the space of all compact operators $T$ on
$\mathcal{H}$ with the singular value sequence
$\{\lambda_{n}\}_{n\in I}$ belonging to $\ell^{p}$. The space
$S_{p}(\mathcal{H})$ is a Banach space with the norm
\begin{eqnarray}\label{norm}
\|T\|_{p}=\left(\sum_{n}|\lambda_{n}|^{p}\right)^{\frac{1}{p}}.
\end{eqnarray}
 The Banach space $S_{1}(\mathcal{H})$ is called the\textit{ trace class} of $\mathcal{H}$. A compact operator $T\in S_{1}(\Hil)$ if and only if $trace(T) := \sum_{i\in I}\left\langle Te_i,e_i\right\rangle<\infty,$ for every orthonormal basis $\{e_i\}_{i\in I}$ for $\Hil$.
Also, $S_{2}(\mathcal{H})$ is called the \textit{Hilbert-Schmidt class} and $T\in S_{2}(\Hil)$ if and only if $\|T\|^2_2=\sum_{i\in I}\left\| Te_i\right\|^2<\infty$.

We know that $T\in S_p(\mathcal{H})$ if and only if
$\{\|{Te_n}\|\}_{n\in I} {\in \ell^p}$ for all orthonormal bases
$\{e_n\}_{n\in I}$. For $0<p \leq 2$ it is even enough to have the
property for only one orthonormal basis.  It is
well known  that $S_p(\mathcal{H})$ is a two sided $*$-ideal of
$\BL{\mathcal{H}}$, that is, a Banach algebra under the norm
(\ref{norm}) and the finite rank operators are dense in
$(S_p(\mathcal{H}), \|.\|_p)$. Moreover, for $T\in
S_p(\mathcal{H})$, one has $\|T\|_p = \|T^*\|_p, \|T\|\leq\|T\|_p$
and if $S\in \BL{\mathcal{H}_1,\mathcal{H}_2}$, then $\|ST\|_p
\leq\|S\|\|T\|_p$ and $\|TS\|_p\leq\|S\|\|T\|_p$.
It is well known that $S_{2}(\Hil_{1},\Hil_{2})$, the space of Hilbert-Schmidt operators from $\Hil_1$ to $\Hil_2$, is a Hilbert space under the following inner product
\begin{eqnarray*}
\left\langle U_1,U_2\right\rangle=trace(U_1^*U_2),\qquad \left(U_1,U_2\in S_{2}(\Hil_{1},\Hil_{2})\right).
 \end{eqnarray*}
  For more
information about these operators, see
\cite{gohberg1,pitch,schatt1,weidm80}.
\begin{lemma}
Let $W=(W_i,w_i)$,  $V=(V_i,v_i)$ be Bessel fusion sequences in $\Hil$ and $O\in \BL{\Hil}$. For $0<p<\infty$, the following are  valid.
\begin{enumerate}
\item\label{1} If $O$ is compact, then $\mathcal{M}^{(W,V)}(O)$ and $\mathcal{M}_{\otimes}^{(W,V)}(O)$ are compact.
\item\label{2} If $O\in S_p(\Hil)$, then $\mathcal{M}^{(W,V)}(O)$, $\mathcal{M}_{\otimes}^{(W,V)}(O) \in S_p\left(\ltiv,\ltiw\right)$.
\end{enumerate}
When $W$ and $V$ are fusion frames, then (1) and (2) are equivalent.
\end{lemma}
\begin{proof}
The sentences (1) and (2) are followed by the ideal property  compact operators and $S_p(\Hil)$ in $\BL{\Hil_1,\Hil_2}$. For compact operators see Theorem 4.18 of [Rudin, 1973].  Conversely, if $W$ and $V$ are fusion frames then
\begin{eqnarray*}
O=S_{W}^{-1}T_WT_{W}^*OT_VT_V^*S_{V}^{-1}=
S_{W}^{-1}T_W\mathcal{M}^{(W,V)}(O)T_V^*S_{V}^{-1}.
\end{eqnarray*}
Again by the ideal property $S_p(\Hil)$ in $\BL{\Hil_1,\Hil_2}$ the result follows immediately.
\end{proof}
 The following theorem computes the pseudo-inverse of $\mathcal{M}^{(W,V)}$, $\mathcal{O}^{(V,W)}$, $\mathcal{M}_{\otimes}^{(W,V)}$ and $\mathcal{O}_{\otimes}^{(V,W)}$ restricted to Hilbert-Schmidt operators, which are denoted by $\mathcal{M}_{\HS}^{(W,V)}$, $\mathcal{O}_{\HS}^{(V,W)}$, $\mathcal{M}_{\otimes \HS}^{(W,V)}$ and $\mathcal{O}_{\otimes \HS }^{(W,V)}$, respectively.
\begin{theorem}
Let $W$ and $V$ be fusion frames in $\Hil$. Then
\begin{enumerate}
\item $\left(\mathcal{M}^{(W,V)}_{\HS}\right)^{\dag}= \left( S_W^{-1}\otimes S_V^{-1}\right)\mathcal{O}_{\HS}^{(V,W)}$.
\item  $\left(\mathcal{M}_{\otimes {\HS}}^{(W,V)}\right)^{\dag}=\mathcal{O}_{\otimes \HS}^{(V,W)}$.
\end{enumerate}
\end{theorem}

\begin{proof}  Obviously, we can see that $\mathcal{M}_{\HS}^{(W,V)}:S_2(\Hil)\rightarrow S_2\left(\ltiv,\ltiw\right)$ given by
\begin{eqnarray*}
\mathcal{M}_{\HS}^{(W,V)}(O)=T_W^*OT_V,
\end{eqnarray*}
for all $O\in S_2(\Hil)$ is well-defined. Also,
$$\mathcal{M}_{\HS}^{(W,V)}\left( S_W^{-1}\otimes S_V^{-1}\right)\mathcal{O}_{\HS}^{(V,W)}\mathcal{M}_{\HS}^{(W,V)}=\mathcal{M}_{\HS}^{(W,V)}.$$
  Moreover, $\left(\mathcal{M}_{\HS}^{(W,V)}\right)^*M=T_WMT_V^*$. Indeed,  for all $M\in  S_2\left(\ltiv,\ltiw\right)$ we have
\begin{eqnarray*}
\left\langle \left(\mathcal{M}_{\HS}^{(W,V)}\right)^*M,O\right\rangle
&=&\left\langle M, T_W^*OT_V\right\rangle\\
&=&trace(M^*T_W^*OT_V)\\
&=&trace(T_VM^*T_W^*O)=\left\langle T_{W}MT_{V}^*,O\right\rangle,
\end{eqnarray*}
for all $O\in S_2(\Hil)$. The third identity follows from the fact that $trace(UT)=trace(TU)$ for all $U\in S_1(\Hil)$ and $T\in B(\Hil)$. Hence,
\begin{eqnarray*}
\left(S_W^{-1}\otimes S_V^{-1}\right)\mathcal{O}_{\HS}^{(V,W)}(M)=S_W^{-1}T_WMT_V^{*}S_V^{-1}=
S_W^{-1}\left(\mathcal{M}_{\HS}^{(W,V)}\right)^*MS_V^{-1}
\end{eqnarray*}
and this immediately follows that
$\range{\mathcal{O}_{\HS}^{(V,W)}}=
\range{\left(\mathcal{M}_{\HS}^{(W,V)}\right)^*}$ and $\kernel{\mathcal{O}_{\HS}^{(V,W)}}=
\kernel{\left(\mathcal{M}_{\HS}^{(W,V)}\right)^*}$. Hence,  $\left(\mathcal{M}_{\HS}^{(W,V)}\right)^{\dag}=\mathcal{O}_{\HS}^{(V,W)}$.
In the same way, (2) follows.
\end{proof}

\subsection{Fusion Frames in Hilbert Schmidt Operators} \label{sec:FusionframeHS}

The above results bear a striking resemblance to the frame-related operators, the analysis and synthesis operators, which we will make more formal below. In the discrete frame case, it is known that the tensor product of frames build a frame again, see \cite{xxlframoper1,xxlframehs07}. For fusion frame we can show a similar result.

 Suppose $W = (W_i,w_i)$ and $V = (V_i,v_i)$ are fusion frames.  Let us denote by $W_j\otimes V_i$ the space of Hilbert-Schmidt operators from $V_i$ to $W_j$. This notation is justified as this space is isomorphic to the tensor product of the two spaces with the inner product introduced above \cite{defant}. This can be considered as a subspace of $S_2(\Hil)$, which we will also denote by $\Hil \otimes \Hil$.
\begin{lemma} Let $W = (W_i,w_i)$ and $V = (V_i,v_i)$ be fusion frames. Then for all $i,j$ we have;
\begin{enumerate}
\item $W_j\otimes V_i=\left\{\pi_{W_j}  O \pi_{V_i},  O \in S_2(\Hil)\right\}$ \label{charaHSji}
\item  $\pi_{W_j}\otimes\pi_{V_i}$ is the orthogonal projection on $W_j\otimes V_i$. i.e.  $\pi_{W_j\otimes V_i}=\pi_{W_j}\otimes\pi_{V_i}$. In particular, $W_j\otimes V_i$ is a closed subspace of $S_2(\Hil)$,
\end{enumerate}
\end{lemma}
\begin{proof} Using the ideal property of $S_2(\Hil)$ immediately follows that $\pi_{W_j} O \pi_{V_i}\in W_j\otimes V_i$, for all  $O \in S_2(\Hil)$. Conversely, assume that $U_{ji}:V_i\rightarrow W_j$ belongs to $W_j\otimes V_i$, then $U_{ji}$ can be extended to a bounded operator $O$ from $\Hil$ to $W_j$. In particular, $O\in S_2(\Hil)$ and $U_{ji}=\pi_{W_j} O \pi_{V_i}$. This proves (1).
To compute the orthogonal projection on $W_j\otimes V_i$ we first apply (\ref{charaHSji}) to show that $ \pi_{W_j}\otimes\pi_{V_i}:S_2(\Hil) \rightarrow W_j \otimes V_i$ is the identity on $W_j\otimes V_i$. In addition, if $U\in (W_j\otimes V_i)^{\perp}$ and $O\in S_2(\Hil)$, then
\begin{eqnarray*}
\langle (\pi_{W_j}\otimes\pi_{V_i}) U, O\rangle&=&\langle \pi_{W_j}U\pi_{V_i}, O\rangle\\
&=&trace \left(\pi_{V_i}U^*\pi_{W_j}O \right)\\
&=&trace \left(U^*\pi_{W_j}O \pi_{V_i}\right)\\
&=&\langle U, \pi_{W_j}O \pi_{V_i}\rangle=0,
\end{eqnarray*}
where the last identity follows from (\ref{charaHSji}).

\end{proof}

We can prove (similar to the results in  \cite{Arefi16,xxlframehs07}) that the tensors of fusions systems are fusion systems in the class of Hilbert-Schmidt operators with the same properties in the following sense:
 \begin{theorem}
The sequences $(W,w_i)$ and $(V,v_i)$ are Bessel fusion sequences (fusion frames, fusion Riesz sequences) with frame bounds $A_W, B_W$ and $A_V, B_V$, respectively if and only if $W\otimes V:=\left( W_j\otimes V_i, v_iw_j\right)$ is a Bessel fusion sequence (fusion frame, fusion Riesz sequence) for $S_2({\Hil})$ with frame bounds $A_V A_W$ and $B_V B_W$.
\end{theorem}

 \begin{proof}
Let $T_{W\otimes V}^*:S_2(\Hil)\rightarrow\sum_{i,j}\bigoplus(W_j\otimes V_i)$ denote the analysis operator of $W\otimes V$. Then
\begin{eqnarray*}
T_{W\otimes V}^*O&=&\left\{w_jv_i\pi_{(W_j\otimes V_i)}O\right\}\\
&=&\left\{w_jv_i\pi_{W_j}O \pi_{V_i}\right\}\\
&=&\left(T_W^*\otimes T_V^*\right)(O),
\end{eqnarray*}
 for all $O\in S_2(\Hil)$. Combining \cite[Theorem 3.12]{caskut04}
and Corollary \ref{Mlr} the result follows immediately .

\end{proof}
Suppose that $W$ and $V$ are Bessel fusion sequences in $\Hil$. Denote by $S_{W\otimes V}$ the fusion frame operator of $W\otimes V$, then
 \begin{eqnarray*}
S_{W\otimes V}&=&T_{W\otimes V}T_{W\otimes V}^*\\
&=&\left(T_{W}\otimes T_V\right)\left(T_W^*\otimes T_V^*\right)\\
&=&S_{W}\otimes S_V.
\end{eqnarray*}
As a consequence, we summarize the basic facts of frames of Hilbert-Schmidt operators as following.

\begin{corollary} Let $(W_i,w_i)$ and $(V_i,v_i)$ be  fusion frames. Then the following assertions for $W\otimes V$ hold:
\begin{enumerate}
\item $S_{W\otimes V}=S_{W}\otimes S_V$ is the frame operator of $W\otimes V$.
\item $S_{W\otimes V}^{-1}=S_W^{-1}\otimes S_V^{-1}$. In particular, the canonical dual of $W\otimes V$ is the tensor product of their canonical
duals.
\item The tensor product of alternate duals of $W$ and $V$ is an alternate dual of $W\otimes V$.
\end{enumerate}
\end{corollary}

\section{Solving Operator Equations}
 Recall that for discrete frames $\Phi$ and $\Psi$, we denote $T_{\Psi}^*OT_{\Phi}$ by $\mathcal{M}^{(\Psi,\Phi)} (O)$ as the Gram-matrix of an operator $O\in B(\Hil)$, also $O^{(\Psi, \Phi)}(M)=T_{\Psi}MT_{\Phi}^*$ indicates the operator induced by a matrix $M\in B(\ell^2)$. By a straightforward  calculation we can prove the following lemma (see \cite{kohl19}):
\begin{lemma}\label{dir}
 Let $\{W_i\}_{i\in I}$ be a family of closed subspaces in $\Hil$ and $U_i\in B(W_i)$, for all $i\in I$ such that $sup_{i\in I}\|U_i\|<\infty$. Then $\oplus U_i:\ltiw\rightarrow \ltiw$ defined by $\oplus U_i\{f_i\}_{i\in I}=\{U_if_i\}_{i\in I}$ is a well-defined bounded operator.
 \end{lemma}
The following result can also be shown in a straightforward way (see also \cite{kohl19}):
 \begin{lemma}\label{sisi}
 Let $W=(W_i,w_i)$ be a fusion frame in $\Hil$ and $\Psi^{(i)}=\{\psi_{i,j}: j \in J_i\}$ be a Riesz basis (resp. frame) for $W_i$, for all $i\in I$ with bounds $A_i$ and $B_i$, respectively such that
 \begin{eqnarray}\label{sha1}
 0<inf_{i\in I}A_i\leq sup_{i\in I}B_i<\infty.
 \end{eqnarray}
Then $\Psi:=\{w_i\psi_{ij}\}_{i\in I,j\in J_i}$ is a Riesz basis (resp. frame). Moreover,
 \begin{enumerate}
 \item[(1)] $T_W \circ \left( \oplus T_{\Psi^{(i)}} \right) = T_\Psi$.
 \item[(2)] $S_\Psi = T_W \circ \left( \oplus S_{\Psi^{(i)}} \right) \circ T_W^*$.
 \end{enumerate}
 \end{lemma}
 \begin{proof}
 Using \cite[Theorem 3.2]{caskut04}, \cite[Proposition 2.4]{sharxxl17}  and the assumptions we have $\Psi$ is a Riesz basis (resp. frame) for $\Hil$. The synthesis operator $T_{\Psi}:\sum\bigoplus \ell^2\rightarrow \Hil$ is defined by $T_{\Psi}\{c_{ij}\}_{i\in I,j\in J_i}=\sum_{i\in I}\sum_{j\in J_i}c_{ij}\psi_{ij}$, where we identify $\ell^2\otimes \ell^2$ by $\sum\bigoplus \ell^2$. Also, (\ref{sha1}) follows that $sup_{i\in I}\|T_{\Psi^{(i)}}\|<\infty.$ Hence, $\oplus T_{\Psi^{(i)}}:\sum\bigoplus \ell^2\rightarrow \ltiw$ is a well-defined and bounded operator by Lemma \ref{dir}.
 \end{proof}
A fusion frame $W=(W_i,w_i)$ with  the local frames $\{\Psi^{(i)}\}_{i\in I}$ satisfied in the Theorem \ref{sisi} is called a fusion frame system and  denoted by $W=(W_i,w_i,\psi_{ij})$. If $\left(W,w_i,\psi_{ij}\right)$ and $\left(V,v_i,\phi_{ij}\right)$ are fusion frame systems, then $\left(W\otimes V,v_iw_i,\psi_{ij}\otimes\phi_{ij}\right)$ is also a fusion frame system for $S_2(\Hil)$.

In combination we get:
\begin{corollary}
Let $W=(W_i,w_i,\Psi_{ij})$ and $V=(V_i,v_i,\Phi_{ij})$ be fusion frame systems (with the same index sets $K=\{(i,j): i \in I, j \in J_i\}$). Then for all $O\in \BL{\Hil}$ and $M\in  \BL{\ell^2}
$ we have
\begin{enumerate}
\item[(1)] $[\mathcal{M}^{(\Psi,\Phi)} (O)]_{(i,j)(k,l)}=[\oplus\mathcal{M}^{(\Psi^{(i)},\Phi^{(k)})} (\mathcal{M}^{(W,V)} (O))]_{(i,j)(k,l)}.$
\item[(2)] $\mathcal O^{(\Psi, \Phi)}(M) = \mathcal O^{(W,V)}[\oplus\mathcal O^{(\Psi^{(i)},\Phi^{(k)})} (M)]_{(i,k)}.$
\end{enumerate}
\end{corollary}
As a result of the above theorem we have
\begin{eqnarray*}
[\mathcal{M}^{(\Psi,\Phi)} (S_W^{-1/2}\otimes S_V^{-1/2})(O)]_{(i,j)(k,l)}=[\oplus\mathcal{M}^{(\Psi^{(i)},{\Phi^{(k)}})} (\mathcal{M}_{\otimes}^{(W,V)} (O))]_{(i,j)(k,l)}.
\end{eqnarray*}
Also,
\begin{eqnarray*}
(S_W^{-1/2}\otimes S_V^{-1/2})\mathcal{O}^{(\Psi,\Phi)}(M)=\mathcal{O}_{\otimes}^{(W,V)}[\oplus \mathcal{O}^{(\Psi^{(i)},\Phi^{(k)})}(M)]_{(i,k)}.
\end{eqnarray*}

In particular, if $S_{\Psi}=S_W$ and $S_{\Phi}=S_V$, for example every local frame $\Psi^{(i)}$ and $\Phi^{(i)}$ are Parseval (see Lemma \ref{sisi}), then
$$[\mathcal{M}_{\otimes}^{(\Psi,\Phi)}  (O)]_{(i,j)(k,l)}=[\oplus\mathcal{M}^{(\Psi^{(i)},{\Phi^{(k)}})} (\mathcal{M}_{\otimes}^{(W,V)} (O))]_{(i,j)(k,l)}. $$
Consider the operator equation
\begin{eqnarray}\label{o}
Of & = & g,
\end{eqnarray}
where $O\in \BL{\Hil}$. For the operator equation (\ref{o}) we can introduce the following system of linear equations
\begin{eqnarray}\mathcal{M}^{(W,W)}(O ) T_W^* f & = &  T_W^* g \label{LinEq1} \\
\mathcal{M}^{(W,W)}(O S_W^{-1}) T_W^*S_W^{-1} f & = &  T_W^* S_W^{-1}g \label{LinEq1'} \\
\mathcal{M}^{(\Psi,\Psi)}(O S_{\Psi}^{-1}) T_{\Psi}^* f & = &  T_{\Psi}^* g \label{LinEq2} \\
\mathcal{M}_{\otimes}^{(W,W)}(O) T_W^* f & = &  T_W^* g \label{LinEq3} \\
\mathcal{M}_{\otimes}^{(\widetilde{W},W)}(O) T_W^* f & = &  T_W^* g \label{LinEq4}.
\end{eqnarray}
Note that, for numerical computations, the above linear systems are solved. It is very natural to ask which representation above is numerically more efficient. Between (\ref{LinEq1}) and (\ref{LinEq1'}), the second is more suitable since $S_W^{-1}$ in the all process of (\ref{LinEq1}) is appear, however, in (\ref{LinEq1'}) it is only  in the final step. Obviously, (\ref{LinEq3}) coincide (\ref{LinEq1}) when the subspaces $W_i,$ for all $i\in I$ are one-dimensional. In general, every equation of (\ref{LinEq1}) is a linear combination of equations of (\ref{LinEq2}).
 Therefore, applying (\ref{LinEq1}) reduces the computation time and iterations. An important distinction is between nonlinear equation (\ref{o}) and linear systems (\ref{LinEq2}),..., (\ref{LinEq4}).
It is not difficult to see that the solutions of (\ref{o}) coincide with the solutions of (\ref{LinEq1}) and (\ref{LinEq2}). Also, there exists a one to one correspondence between solutions of (\ref{o}) and solutions of (\ref{LinEq3}) and (\ref{LinEq4}).

\subsection{Block matrices by fusion frames} \label{sec:blockfus}

By  \eqref{tensor} we have a block-matrix representation. Assuming finite frames \cite{xxlfinfram1,caku13} (just for this representation)  we can write for $U_{i,j} = \left(\mathcal{M}^{(W, W)}(O)\right)_{ i,j}$:
\begin{equation} \label{sec:blockstruct1} O =  T_{\w} \phi_{\w W} \left( \begin{array}{c c c c} U_{0,0} & U_{0,1} & \cdots & U_{0,b-1} \\ U_{1,0} & U_{1,1} & \cdots & U_{1,b-1} \\ \vdots & \vdots & \ddots & \vdots\\ U_{b-1,0} & U_{b-1,1} & \cdots & U_{b-1,b-1} \end{array} \right) \phi_{\w W}T^*_{\w},
\end{equation}




\section{Applications}

In this section, we show that existing algorithms and operator representations can
be interpreted as fusion frames representations. In particular, we focus on
convolution operators, where it is shown that the degrees of freedom
obtained by using redundant fusion frames can be used to obtain sparse and structured
representations of convolutions.

\subsection{Overlapped convolution algorithms}

We consider convolution operators in $L^2(\mathbb R)$. For a function $h \in L^1(\mathbb R)$, the convolution
operator $O$ is defined by
\begin{eqnarray*}
   O(f)(t) &= & (h \star f)(t)\\
   & = & \int_{\mathbb{R}} f(t-u) h(u) du.
\end{eqnarray*}
Boundedness of $O$ is guaranteed by the Young inequality, i.e. $\|O\|_{L^2(\mathbb R) \rightarrow L^2(\mathbb R)} \leq \|h\|_{L^1(\mathbb R)}$.

Fusion frames can be used to find efficient representations of a convolution
operator as a matrix of convolution operators on bounded intervals. Numerically, such convolutions
can be efficiently computed using Fast Fourier Transforms.

At first fusion frame representation of $O$ is obtained by considering the
fusion orthogonal basis $(W_i)_{i\in\mathbb Z}$ where $W_i = L^2(i, i+1)$, with orthogonal projections
$\pi_{W_i}(f) = f\mathbf{1}_{(i, i+1)}$.

In this orthogonal fusion basis, the convolution $O$ is represented by the matrix of operators
$\mathcal{M}^{(W,W)}(H)$ with  (see \eqref{tensor})
\begin{eqnarray*}
\left(\mathcal{M}^{(W,W)}(O)\right)_{ j,i} & = &  \pi_{W_i} H \pi_{W_j} \\
& = &  O_{ij}
\end{eqnarray*}
where
\begin{eqnarray*}
 O_{ j,i}(f) & = & \mathbf{1}_{[i, i+1]} \left(h \star (\mathbf{1}_{[j, j+1]} f) \right).
\end{eqnarray*}
In the definition of $ O_{j,i}$, $h$ can be replaced by $h_{i-j}$ with $h_k =  \mathbf{1}_{[k-1, k+1]} h$. Furthermore, $ O_{ij}(f)$ can be obtained
by computing a circular convolution: without loss of generality, for $j=1$ and $i=0$,
 $ O_{10}(f)$ is equal, in the interval $(1,2)$, to  the circular convolution
 of $h_1$ and $\mathbf{1}_{[0, 1]} f$, computed in the interval $(0,2)$. As the Fourier coefficients of the circular convolution are given by the products of the Fourier coefficients of $h_1$ and $\mathbf{1}_{[0, 1]} f$ in the interval $(0,2)$, $H_{10}$ is a frame multiplier \cite{xxlmult1} $M_{m, \Phi, \Psi}$, where the symbol $m$ is given by the Fourier coefficients of $h_1$, $\Phi = (\phi_n)_{n\in\mathbb Z}$ and $\Psi = (\psi_n)_{n\in\mathbb Z}$ are  frames
 on $L^2(0,1)$ and $L^2(1,2)$ resp., with $\phi_n(t) = \psi_n(t) = \exp(i\pi nt) / \sqrt{2}$.

We consider now the case where the support of $h$ is the interval $[0,L]$, with finite $L$.
We show that the Overlap-Add and Overlap-Save algorithms \cite{stockham, oppenheim}
for fast numerical convolutions can be interpreted as fusion frame matrix representations of the convolution operator $ O$. These algorithms compute convolutions of a long signal with a short impulse response by splitting the signal in short segments, and computing convolutions of signals with finite support using the Fast Fourier Transform.
Here, we take advantage of the non-uniqueness of the fusion frame representation of an operator to get a simple representation of a convolution operator.

Overlap-add consists in decomposing $f$ using the frame $W$ as $\sum_{i\in\mathbb Z} f\mathbf{1}_{(i, i+1)}$, and applying $H$ on each term:
\begin{equation*}
h \star f = \sum_{i\in\mathbb Z} h \star (f\mathbf{1}_{(i, i+1)})
\end{equation*}
As the support of $h \star f\mathbf{1}_{(i, i+1)}$ is included in $[i, i+L+1]$,
each term of the sum is in a subspace of the fusion frame $V$, with $V_i = L^2(i, i+L+1)$, and
the convolution $H$ can be represented as
\begin{equation*}
     O = \mathcal{O}^{(V,W)}( \mathcal{M}^{oa})
\end{equation*}
with $ \mathcal{M}^{oa}$ the matrix of operators such that $ \mathcal{M}^{oa}_{ij} = 0$ for $i\neq j$, and
$ \mathcal{M}^{oa}_{ii} : W_i \rightarrow V_i$ is defined by
$ \mathcal{M}^{oa}_{ii}(f) = h\star f$.

Similarly, the Overlap-Save algorithm is obtained by using the fusion frames $V$ and $U$, with $U_i=L^2(i+L, i+L+1)$. Then
$H = \mathcal{O}^{(U,V)}(H^{os})$
with $H^{os}_{ij} = 0$ for $i\neq j$, and
\begin{equation*}
    \mathcal{M}^{os}_{ii}(f) =
   \mathbf{1}_{[i+L, i+L+1]}(h\star f)
\end{equation*}

In both cases, the convolution operator is decomposed as a diagonal matrix of operators, where the
operators on the diagonal are convolutions on bounded intervals, that can be computed using Fourier series.

\subsection{Non-standard representation of operators in wavelet frames}

In this second example, we give a fusion frame interpretation of the non-standard representation of operators in wavelet bases \cite{bey92}.

We consider a multiscale analysis, i.e., a
dense, nested sequence of finite dimensional subspaces
\[
 \dots \subseteq V_{2} \subset V_{1}  \subseteq V_0  \subseteq V_{-1}  \subseteq V_{-2}   \subseteq \dots  \subseteq V_j
  	 \subseteq\dots \subseteq L^2(\RR),
\]
such that
\[
L^2(\RR) = \overline{\bigcup_{j\in\mathbb{N}_0}V_j},
  	\qquad V_0 = \bigcap_{j\in\mathbb{N}_0}V_j.
\]

We define $W_j$ as the orthogonal complement of $V_j$ in $V_{j-1}$,
$$ V_{j-1} = V_j \oplus W_j.$$
Then $L^2(\RR)  = \bigoplus W_j$.

As in \cite{bey92} let us consider $\{ \psi_{j,k} \}$ an orthonormal basis for $W_j$ and $\{ \varphi_{j,k} \}$ and ONB for $V_j$. $P_j$ and $Q_j$ are the orthogonal projections in $V_j$ and $W_j$, respectively.

The space $V_0$ of the approximations of functions of $L^2(\mathbb R)$ at scale 0 can be decomposed as
a space of coarse approximations $V_n$, and detail subspaces $W_j$:
\begin{eqnarray*}
V_0 = V_n + \sum_{j=1}^n W_j.
\end{eqnarray*}

An approximation $T_0$ of an operator at scale 0 is given by $T_0 = P_0TP_0$.
Using this fusion orthogonal basis $(V_n, W_n,  \ldots, W_2, W_1)$ to represent the operator $T_0$ would involve operators between spaces at different scales.

Following \cite{bey92}, we use the nonstandard form of an operator $T$, i.e. the set of triplets $\{A_j, B_j, \Gamma_j\}_{j\in\mathbb Z}$, where
\begin{eqnarray*}
A_j & = & Q_j T Q_j \\
B_j & = & Q_j T P_j \\
\Gamma_j & = & P_j T Q_j
\end{eqnarray*}

Using the nonstandard form of an operator allows to decompose the operator
using operators between spaces at the same scales :
\begin{equation*}
T_0 = T_n + \sum_{j=1}^n \left(A_j + B_j + \Gamma_j\right).
\end{equation*}

This can be interpreted as a fusion frame representation of $T_0$ by considering the
fusion frame $F_n = (V_n, V_{n-1}, \ldots, V_1, W_n, W_{n-1}, \ldots, W_1)$, and the matrix representation
\begin{equation*}
    T_0 = \mathcal{O}^({F_n, F_n})(M)
\end{equation*}
with
\begin{equation*}
M =
\left(
\begin{array}{cccc:cccc}

T_n & & & & \Gamma_{n}\\
 & & & &  & \Gamma_{n-1} \\
 & & & &  & & \ddots \\
 & & & &  & & & \Gamma_1 \\
\hdashline
 B_{n} & & & & A_{n} \\
 & B_{n-1} & & & & A_{n-1} \\
 & & \ddots & & & & \ddots \\
 & & & B_{1} & & & & A_{1} \\
\end{array}
\right)
\end{equation*}

In the case where $T$ is  an integral operator
of the form
\begin{equation*}
    T(f)(x) = \int_{\mathbb R} K(x,y) f(y) dy,
\end{equation*}
then the coefficients of the matrix representations
of $A_n$, $B_n$ and $\Gamma_n$ using the orthogonal bases of $W_n$ and $V_n$ are given the 2D wavelets coefficients of $K(x, y)$ \cite{beycoif91}.

Furthermore, if $T$ is translation-invariant (i.e. $K(x, y) = h(x-y))$, the representation of $A_j$ in the
wavelet basis of $W_j$ is given by the coefficients
\begin{eqnarray*}
 \alpha_{il}^j & = & \int_{-\infty}^{+\infty} \psi_{j,i}(x) (T\psi_{j,l})(x)dx  \\
& =& \int_{-\infty}^{+\infty} \psi_{j,i-l}(x) (T\psi_{j,0})(x)dx  \\
\end{eqnarray*}
which has a Toeplitz structure.
Similar structures are found for the matrix representations
of $B_j$ and $\Gamma_j$.

\small
\bibliographystyle{abbrv}

\end{document}